\documentclass[preprint,12pt]{elsarticle}

\usepackage{amssymb}
\usepackage{amsthm}

\newcommand{\sysn}{\left\{\begin{array}{rcl}}
\newcommand{\sysk}{\end{array}\right.}

\newtheorem{theorem}{Theorem}[section]
\newtheorem{lemma}[theorem]{Lemma}

\theoremstyle{example}
\newtheorem{example}[theorem]{Example}

\newtheorem{proposition}[theorem]{Proposition}
\theoremstyle{definition}
\newtheorem{definition}[theorem]{Definition}
\newtheorem{remark}[theorem]{Remark}
\newtheorem{corollary}[theorem]{Corollary}

\journal{...}

\begin{document}

\begin{frontmatter}



\title{Selection principles in function spaces with the compact-open topology}


\author{Alexander V. Osipov}


\ead{OAB@list.ru}


\address{Krasovskii Institute of Mathematics and Mechanics, Ural Federal
 University,

 Ural State University of Economics, 620219, Yekaterinburg, Russia}

\begin{abstract}
For a Tychonoff space $X$, we denote by $C_k(X)$ the space of all
real-valued continuous functions on $X$ with the compact-open
topology. A subset $A\subset X$ is said to be sequentially dense
in $X$ if every point of $X$ is the limit of a convergent sequence
in $A$. In this paper, the following properties for $C_k(X)$ are
considered.

\medskip
\begin{center}
$S_1(\mathcal{S},\mathcal{S}) \Rightarrow
S_{fin}(\mathcal{S},\mathcal{S}) \Rightarrow
S_1(\mathcal{S},\mathcal{D}) \Rightarrow
S_{fin}(\mathcal{S},\mathcal{D})$ \\  \, \, $\Uparrow$ \, \, \, \,
\,\, \,  \, \, \, $ \Uparrow $ \,\, \, \, \,\, \, \,
\, \, \, $\Uparrow $ \,\, \, \, \,\, \, \,\, \, \, $\Uparrow$ \\
$S_1(\mathcal{D},\mathcal{S}) \Rightarrow
S_{fin}(\mathcal{D},\mathcal{S}) \Rightarrow
S_1(\mathcal{D},\mathcal{D}) \Rightarrow
S_{fin}(\mathcal{D},\mathcal{D})$

\end{center}

\bigskip

For example, a space $C_k(X)$ satisfies
$S_1(\mathcal{S},\mathcal{D})$ (resp.,
$S_{fin}(\mathcal{S},\mathcal{D}))$ if whenever $(S_n : n\in
\mathbb{N})$ is a sequence of sequentially dense subsets of
$C_k(X)$, one can take points $f_n\in S_n$ (resp., finite
$F_n\subset S_n$) such that $\{f_n : n\in \mathbb{N}\}$ (resp.,
$\bigcup \{F_n: n\in \mathbb{N}\}$) is dense in $C_k(X)$. Other
properties are defined similarly.

In \cite{os2}, we obtained characterizations these selection
properties for $C_p(X)$. In this paper, we have gave
characterizations for $C_k(X)$.

\end{abstract}


\end{frontmatter}



\section{Introduction}
\label{}

For a Tychonoff space $X$, we denote by $C_k(X)$ the space of all
real-valued continuous functions on $X$ with the compact-open
topology. Subbase open sets of $C_k(X)$ are of the form $[A,
U]=\{f\in C(X): f(A)\subset U \}$, where $A$ is a compact subset
of $X$  and $U$ is a non-empty open subset of $\mathbb{R}$. Since
the compact-open topology coincides with the topology of uniform
convergence on compact subsets of $X$, we can represent a basic
neighborhood of the point $f\in C(X)$ as $<f,A,\epsilon>$ where
$<f,A,\epsilon>:=\{g\in C(X): |f(x)-g(x)|<\epsilon$ $\forall$
$x\in A\}$, $A$ is a compact subset of $X$ and $\epsilon>0$.

Many topological properties are defined or characterized in terms
 of the following classical selection principles.
 Let $\mathcal{A}$ and $\mathcal{B}$ be sets consisting of
families of subsets of an infinite set $X$. Then:

$S_{1}(\mathcal{A},\mathcal{B})$ is the selection hypothesis: for
each sequence $(A_{n}: n\in \mathbb{N})$ of elements of
$\mathcal{A}$ there is a sequence $\{b_{n}\}_{n\in \mathbb{N}}$
such that for each $n$, $b_{n}\in A_{n}$, and $\{b_{n}:
n\in\mathbb{N} \}$ is an element of $\mathcal{B}$.

$S_{fin}(\mathcal{A},\mathcal{B})$ is the selection hypothesis:
for each sequence $(A_{n}: n\in \mathbb{N})$ of elements of
$\mathcal{A}$ there is a sequence $\{B_{n}\}_{n\in \mathbb{N}}$ of
finite sets such that for each $n$, $B_{n}\subseteq A_{n}$, and
$\bigcup_{n\in\mathbb{N}}B_{n}\in\mathcal{B}$.

$U_{fin}(\mathcal{A},\mathcal{B})$ is the selection hypothesis:
whenever $\mathcal{U}_1$, $\mathcal{U}_2, ... \in \mathcal{A}$ and
none contains a finite subcover, there are finite sets
$\mathcal{F}_n\subseteq \mathcal{U}_n$, $n\in \mathbb{N}$, such
that $\{\bigcup \mathcal{F}_n : n\in \mathbb{N}\}\in \mathcal{B}$.

\medskip
The following prototype of many classical properties is called
"$\mathcal{A}$ choose $\mathcal{B}$" in \cite{tss}.

${\mathcal{A}\choose\mathcal{B}}$ : For each $\mathcal{U}\in
\mathcal{A}$ there exists $\mathcal{V}\subseteq \mathcal{U}$ such
that $\mathcal{V}\in \mathcal{B}$.

Then $S_{fin}(\mathcal{A},\mathcal{B})$ implies
${\mathcal{A}\choose\mathcal{B}}$.

In this paper, by a cover we mean a nontrivial one, that is,
$\mathcal{U}$ is a cover of $X$ if $X=\bigcup \mathcal{U}$ and
$X\notin \mathcal{U}$.

An open cover $\mathcal{U}$ of a space $X$ is called:

$\bullet$  an $\omega$-cover (a $k$-cover) if each finite
(compact) subset $C$ of $X$ is contained in an element of
$\mathcal{U}$ and $X\notin \mathcal{U}$ (i.e. $\mathcal{U}$ is a
non-trivial cover);

$\bullet$  a $\gamma$-cover (a $\gamma_k$-cover) if $\mathcal{U}$
is infinite, $X\notin \mathcal{U}$, and for each finite (compact)
subset $C$ of $X$ the set $\{U\in \mathcal{U} : C\nsubseteq U\}$
is finite.

A space $X$ is said to be a $\gamma_k$-set if each $k$-cover
$\mathcal{U}$ of $X$ contains a countable set $\{U_n : n\in
\mathbb{N}\}$ which is a $\gamma_k$-cover of $X$ \cite{koc1}.

In a series of papers it was demonstrated that $\gamma$-covers and
$k$-covers play a key role in function spaces
\cite{koc,koc1,koc2,mkn,mc, os2,os4,os3,pp,sak1} and many others.
We continue to investigate applications of $k$-covers in function
spaces with the compact-open topology.

\section{Main definitions and notation}

 If $X$ is a space and $A\subseteq X$, then the sequential closure of $A$,
 denoted by $[A]_{seq}$, is the set of all limits of sequences
 from $A$. A set $D\subseteq X$ is said to be sequentially dense
 if $X=[D]_{seq}$. A space $X$ is called sequentially separable if
 it has a countable sequentially dense set.  Call $X$ strongly sequentially separable,
  if $X$ is separable and  every countable dense subset of $X$ is sequentially dense.
 Clearly, every strongly sequentially separable space is
 sequentially separable, and every sequentially separable space is
 separable.

\medskip
For a topological space $X$ we denote:

$\bullet$ $\mathcal{O}$ --- the family of open covers of $X$;

$\bullet$ $\Gamma$ --- the family of open $\gamma$-covers of $X$;

$\bullet$ $\Gamma_k$ --- the family of open $\gamma_k$-covers of
$X$;

$\bullet$ $\Omega$ --- the family of open $\omega$-covers of $X$;

$\bullet$ $\mathcal{K}$ --- the family of open $k$-covers of $X$;

$\bullet$ $\mathcal{K}_{cz}^{\omega}$ --- the family of countable
co-zero $k$-covers of $X$;

$\bullet$ $\mathcal{D}$ --- the family of dense subsets of
$C_k(X)$;

$\bullet$ $\mathcal{D}^{\omega}$ --- the family of countable dense
subsets of $C_k(X)$;

$\bullet$ $\mathcal{S}$ --- the family of sequentially dense
subsets of $C_k(X)$;

$\bullet$ $\mathbb{K}(X)$ --- the family of all non-empty compact
subsets of $X$.

\bigskip

$\bullet$ A space $X$ is $R$-separable, if $X$ satisfies
$S_1(\mathcal{D}, \mathcal{D})$ (Def. 47, \cite{bbm1}).

$\bullet$ A space $X$ is $M$-separable (selective separability),
if $X$ satisfies $S_{fin}(\mathcal{D}, \mathcal{D})$.

$\bullet$ A space $X$ is selectively sequentially separable, if
$X$ satisfies $S_{fin}(\mathcal{S}, \mathcal{S})$ (Def. 1.2,
\cite{bc}).

\medskip

For a topological space $X$ we have the next relations of
selectors for sequences of dense sets of $X$.

\medskip
\begin{center}
$S_1(\mathcal{S},\mathcal{S}) \Rightarrow
S_{fin}(\mathcal{S},\mathcal{S}) \Rightarrow
S_1(\mathcal{S},\mathcal{D}) \Rightarrow
S_{fin}(\mathcal{S},\mathcal{D})$ \\  \, \, $\Uparrow$ \, \, \, \,
\,\, \,  \, \, \, $ \Uparrow $ \,\, \, \, \,\, \, \,
\, \, \, $\Uparrow $ \,\, \, \, \,\, \, \,\, \, \, $\Uparrow$ \\
$S_1(\mathcal{D},\mathcal{S}) \Rightarrow
S_{fin}(\mathcal{D},\mathcal{S}) \Rightarrow
S_1(\mathcal{D},\mathcal{D}) \Rightarrow
S_{fin}(\mathcal{D},\mathcal{D})$

\end{center}

\bigskip

Let $X$ be a topological space, and $x\in X$. A subset $A$ of $X$
{\it converges} to $x$, $x=\lim A$, if $A$ is infinite, $x\notin
A$, and for each neighborhood $U$ of $x$, $A\setminus U$ is
finite. Consider the following collection:

$\bullet$ $\Omega_x=\{A\subseteq X : x\in \overline{A}\setminus
A\}$;

$\bullet$ $\Gamma_x=\{A\subseteq X : x=\lim A\}$.

Note that if $A\in \Gamma_x$, then there exists $\{a_n\}\subset A$
converging to $x$. So, simply $\Gamma_x$ may be the set of
non-trivial convergent sequences to $x$.

\bigskip

We write $\Pi (\mathcal{A}_x, \mathcal{B}_x)$ without specifying
$x$, we mean $(\forall x) \Pi (\mathcal{A}_x, \mathcal{B}_x)$.

\bigskip

 So we have three types of topological properties
described through the selection principles:

$\bullet$  local properties of the form $S_*(\Phi_x,\Psi_x)$;

$\bullet$  global properties of the form $S_*(\Phi,\Psi)$;

$\bullet$  semi-local of the form $S_*(\Phi,\Psi_x)$.

\medskip

Our main goal is to describe the topological properties for
sequences of dense sets of $C_k(X)$  in terms of selection
principles of $X$.

\section{$S_{1}(\mathcal{D},\mathcal{S})$}

Recall that $X$ a $\gamma'_k$-set if it satisfies the selection
hypothesis  $S_{1}(\mathcal{K},\Gamma_{k})$ \cite{koc1}.

\begin{theorem}(\cite{llt})\label{th11} For a Tychonoff space $X$ the following statements are
equivalent:

\begin{enumerate}

\item  $C_k(X)$ satisfies $S_{1}(\Omega_{\bf 0},\Gamma_{\bf 0})$
(i.e., $C_k(X)$ is strongly Fr$\acute{e}$chet-Urysohn);

\item  $X$ is a $\gamma'_k$-set.

\end{enumerate}

\end{theorem}

\medskip
Recall that the $i$-weight $iw(X)$ of a space $X$ is the smallest
infinite cardinal number $\tau$ such that $X$ can be mapped by a
one-to-one continuous mapping onto a Tychonoff space of the weight
not greater than $\tau$.

\medskip

\begin{theorem} (Noble \cite{nob}) \label{th31}  A space $C_{k}(X)$ is separable iff
$iw(X)=\aleph_0$.
\end{theorem}

\medskip

\begin{theorem}\label{th1} For a Tychonoff space $X$ with $iw(X)=\aleph_0$ the following statements are
equivalent:

\begin{enumerate}

\item  $C_k(X)$ satisfies $S_{1}(\mathcal{D},\mathcal{S})$;

\item   Every dense subset of $C_k(X)$ is sequentially dense;

\item  $X$ satisfies $S_{1}(\mathcal{K},\Gamma_{k})$ ($X$ is a
$\gamma'_k$-set);

\item $X$ is a $\gamma_k$-set;

\item $C_k(X)$ is Fr$\acute{e}$chet-Urysohn;

\item $C_k(X)$ satisfies $S_{fin}(\mathcal{D},\mathcal{S})$;

\item  $X$ satisfies $S_{fin}(\mathcal{K},\Gamma_{k})$;

\item Each finite power of $X$ satisfies
$S_{fin}(\mathcal{K},\Gamma_{k})$;

\item $C_k(X)$ satisfies $S_{1}(\Omega_{\bf 0},\Gamma_{\bf 0})$;

\item $C_k(X)$ satisfies $S_{1}(\mathcal{D},\Gamma_{\bf 0})$.

\end{enumerate}

\end{theorem}

\begin{proof}

$(1)\Rightarrow(6)$ is immediate.

$(4)\Leftrightarrow(5)$. By Theorem 4.7.4 in \cite{mcnt}.

$(3)\Leftrightarrow(4)$. By Theorem 18 in \cite{cmkm}.

$(3)\Leftrightarrow(7)$. By Theorem 5 in \cite{koc1}.

$(3)\Leftrightarrow(8)$. By Theorem 7 in \cite{koc1}.

$(3)\Leftrightarrow(9)$. By Theorem \ref{th11}.

$(9)\Rightarrow(10)$ is immediate.

$(6)\Rightarrow(2)$. Let $D$ be a dense subset of $C_k(X)$. By
$S_{fin}(\mathcal{D},\mathcal{S})$, for
 sequence $(D_i : D_i=D$ and $i\in \mathbb{N} )$
there is a sequence $(K_{i}: i\in\mathbb{N})$ such that for each
$i$, $K_{i}$ is finite, $K_{i}\subset D_{i}$, and
$\bigcup_{i\in\mathbb{N}}K_{i}$ is a countable sequentially dense
subset of $C_k(X)$. It follows that $D$ is a sequentially dense
subset of $C_k(X)$.

$(2)\Rightarrow(4)$. Let $\mathcal{U}$ be an open $k$-cover of
$X$. Note that the set $\mathcal{D} :=\{ f\in C(X) :
f\upharpoonright (X\setminus U)\equiv 1$ for some
$U\in\mathcal{U}\}$ is dense in $C_k(X)$, hence, it is
sequentially dense. Take $f_n\in D$ such that $f_n\mapsto \bf{0}$.
Let $f_n\upharpoonright (X\setminus U_n)\equiv 1$ for some
$U_n\in\mathcal{U}$. Then $\{U_n: n\in \mathbb{N}\}$ is a
$\gamma_k$-subcover of $\mathcal{U}$, because of $f_n\mapsto
\bf{0}$. Hence, $X$ is a $\gamma_k$-set.

$(3)\Rightarrow(1)$. Let $(D_{i,j} : i,j\in \mathbb{N})$ be a
sequence of dense subsets of $C_k(X)$ and let $D=\{f_i : i\in
\mathbb{N}\}$ be a countable dense subset of $C_k(X)$.

For every $i,j\in \mathbb{N}$ consider
$\mathcal{U}_{i,j}=\{U_{h,i,j} :
U_{h,i,j}=(f_i-h)^{-1}(-\frac{1}{j},\frac{1}{j})$ for $h\in
D_{i,j}\}$. Note that $\mathcal{U}_{i,j}$ is an $k$-cover of $X$
for every $i,j\in \mathbb{N}$. Since $X$ a $\gamma'_k$-set, there
is a sequence $(U_{h(i,j),i,j} : i,j\in \mathbb{N})$ such that
$U_{h(i,j),i,j}\in \mathcal{U}_{i,j}$, and $\{U_{h(i,j),i,j}:
i,j\in\mathbb{N} \}$ is an element of $\Gamma_k$. Claim that
$\{h(i,j) : i,j\in\mathbb{N}\}$ is a dense subset of $C_k(X)$. Fix
$g\in C(X)$ and a base neighborhood $W=<g, A, \epsilon>$ of $g$,
where $A$ is a compact subset of $X$ and $\epsilon>0$. There are
$f_i\in D$  and $j\in \mathbb{N}$ such that
$<f_i,A,\frac{1}{j}>\subseteq W$. Since $\{U_{h(i,j),i,j}:
i,j\in\mathbb{N} \}$ is an element of $\Gamma_k$, there is $j'>j$
such that $A\subset U_{h(i,j'),i,j'}$, hence, $h(i,j')\in
<f_i,A,\frac{1}{j'}>\subseteq <f_i,A,\frac{1}{j}> \subseteq W$.

Since $C_k(X)$ is Fr$\acute{e}$chet-Urysohn, every dense subset of
$C_k(X)$ is sequentially dense. It follows that $\{h(i,j) :
i,j\in\mathbb{N}\}$ is sequentially dense.

$(10)\Rightarrow(3)$. Let $\{\mathcal{U}_i: i\in
\mathbb{N}\}\subset \mathcal{K}$ and let $D=\{d_j: j\in
\mathbb{N}\}$ be a countable dense subset of $C_k(X)$. Consider
$D_i=\{f_{K,U,i,j}\in C(X):$ such that $f_{K,U,i,j}\upharpoonright
K\equiv d_j$, $f_{K,U,i,j}\upharpoonright (X\setminus U)\equiv 1$
where $K\in \mathbb{K}(X)$, $K\subset U\in\mathcal{U}_i\}$ for
every $i\in \mathbb{N}$. Since $D$ is a dense subset of $C_k(X)$,
then $D_i$ is a dense subset of $C_k(X)$ for every $i\in
\mathbb{N}$. By (10), there is a set $\{f_{K(i),U(i),i,j(i)} :
i\in \mathbb{N}\}$ such that $f_{K(i),U(i),i,j(i)}\in D_i$ and
$\{f_{K(i), U(i), i, j(i)} : i\in \mathbb{N}\}\in \Gamma_{\bf 0}$.
Claim that a set $\{U(i) : i\in \mathbb{N}\}\in \Gamma_{k}$. Let
$K\in \mathbb{K}(X)$ and let $W=[K,(-\frac{1}{2},\frac{1}{2})]$ be
a base neighborhood of ${\bf 0}$. Since $\{f_{K(i), U(i), i, j(i)}
: i\in \mathbb{N}\}\in \Gamma_{\bf 0}$, there is $i'\in
\mathbb{N}$ such that $f_{K(i), U(i), i, j(i)}\in W$ for every
$i>i'$. It follows that $K\subset U(i)$ for every $i>i'$ and,
hence, $\{U(i) : i\in \mathbb{N}\}\in \Gamma_{k}$.

\end{proof}

Let $S\subset \mathbb{K}(X)$. An open cover $\mathcal{U}$ of a
space $X$ will be call:

$\bullet$ a $s$-cover if each $C\in S$ is contained in an element
of $\mathcal{U}$ and $X\notin \mathcal{U}$;

$\bullet$  a $\gamma_s$-cover if $\mathcal{U}$ is infinite,
$X\notin \mathcal{U}$, and for each $C\in S$ the set $\{U\in
\mathcal{U} : C\nsubseteq U\}$ is finite.

\begin{definition} Let $S\subset \mathbb{K}(X)$. A space $X$ is called a $\gamma_{s}$-set if each $s$-cover of $X$ contains a sequence which is a $\gamma_s$-cover of $X$.

\end{definition}

\medskip

\begin{definition} A space $X$ will be call a $\gamma^{\omega}_k$-set if each
countable cozero $k$-cover $\mathcal{U}$ of $X$ contains a set
$\{U_n : n\in \mathbb{N}\}$ which is a $\gamma_k$-cover of $X$.
\end{definition}
\medskip

For a mapping $f:X \mapsto Y$ we will denote by $f(k)=\{f(K) :
K\in \mathbb{K}(X)\}$.

\medskip

\begin{theorem}\label{th2} For a space $X$ with $iw(X)=\aleph_0$, the following statements are
equivalent:

\begin{enumerate}

\item  $C_k(X)$ satisfies
$S_{1}(\mathcal{D}^{\omega},\mathcal{S})$;

\item  $C_k(X)$ is strongly sequentially separable;

\item $X$ is a $\gamma^{\omega}_k$-set;

\item $X$ satisfies $S_{1}(\mathcal{K}_{cz}^{\omega},
\Gamma_{k})$;

\item for every a condensation (one-to-one continuous mapping) $f:
X \mapsto Y$ from the space $X$ on a
 separable metric space $Y$, the space $Y$ is $\gamma_{f(k)}$-set.

\end{enumerate}

\end{theorem}

\begin{proof}

$(3)\Rightarrow(5)$. Let $f$ be a condensation $f: X \mapsto Y$
from the space $X$ on a
 separable metric space $Y$. If $\mu$ is a $f(k)$-cover of $Y$,
 then there is $\mu'\subset \mu$ such that $\mu'$ is a $f(k)$-cover of $Y$ and $|\mu'|=\aleph_0$.
The family $f^{-1}(\mu')=\{f^{-1}(V) : V\in \mu'\}$ is a countable
co-zero $k$-cover of  $X$. By the argument that $X$ is a
$\gamma^{\omega}_k$-set, we have that $Y$ is $\gamma_{f(k)}$-set.

The remaining implications follow from the proofs of Theorem
\ref{th1} and Theorem 18 in \cite{cmkm}.

\end{proof}

\begin{corollary} For a separable metrizable space $X$,  the following statements are
equivalent:
\begin{enumerate}

\item  $C_k(X)$ satisfies $S_{1}(\mathcal{D},\mathcal{S})$;

\item   Every dense subset of $C_k(X)$ is sequentially dense;

\item  $C_k(X)$ is strongly sequentially separable;

\item $C_k(X)$ is a Fr$\acute{e}$chet-Urysohn;

\item $C_k(X)$ is metrizable and separable;

\item $X$ satisfies $S_{1}(\mathcal{K},\Gamma_{k})$;

\item $X$ satisfies $S_{1}(\mathcal{K}, \mathcal{K})$;

\item $X$ satisfies $S_{fin}(\mathcal{K}, \mathcal{K})$;

\item $X$ is a hemicompact.
\end{enumerate}

\end{corollary}

\begin{proof} By Theorem \ref{th1} and  Theorem 6
in \cite{cmkm}.

\end{proof}

A space $X$ is called a {\it $k$-Lindel$\ddot{o}$f space} if for
each open $k$-cover $\mathcal{U}$ of $X$ there is a
$\mathcal{V}\subseteq \mathcal{U}$ such that $\mathcal{V}$ is
countable and $\mathcal{V}\in \mathcal{K}$. Each
$k$-Lindel$\ddot{o}$f space is Lindel$\ddot{o}$f, so normal, too.

\begin{lemma}(\cite{mcnt})\label{lem1}
$C_k(X)$ has countable tightness iff $X$ is $k$-Lindel$\ddot{o}$f.
\end{lemma}

By Theorem \ref{th1}, Theorem \ref{th2} and Lemma \ref{lem1} we
have

\begin{theorem} For a Tychonoff space $X$ with $iw(X)=\aleph_0$ the following statements are
equivalent:

\begin{enumerate}

\item $C_k(X)$ is Fr$\acute{e}$chet-Urysohn;

\item $C_k(X)$ is strongly sequentially separable and has
countable tightness;

\item  $X$ satisfies $S_{1}(\mathcal{K}^{\omega},\Gamma_{k})$ and
is $k$-Lindel$\ddot{o}$f;

\item   Every dense subset of $C_k(X)$ contains a countable
sequentially dense subset of $C_k(X)$.

\end{enumerate}

\end{theorem}

\bigskip
In Doctoral Dissertation, A.J. March considered the following
problem (Problem 117 in \cite{ma}):  Is it possible to find a
space $X$ such that $C_k(X)$ is strongly sequentially separable
but $C_k(X)^2$ is not strongly sequentially separable ?

We get a negative answer to this question.

\begin{proposition}\label{pr1} Suppose $X$ has property
$S_{1}(\mathcal{K}_{cz}^{\omega},\Gamma_k)$. Then $X\bigsqcup X$
has property $S_{1}(\mathcal{K}_{cz}^{\omega},\Gamma_k)$.
\end{proposition}

\begin{proof} Let $\mathcal{U}=\{U_i: i\in \mathbb{N} \}$ be a countable $k$-cover of $X\bigsqcup X$ by
cozero sets. Let $X\bigsqcup X=X_1\bigsqcup X_2$ where $X_i=X$ for
$i=1,2$. Consider $\mathcal{V}_1=\{U^1_i=U_i\bigcap X_1 :
X_1\setminus U_i\neq\emptyset, i\in \mathbb{N}\}$ and
$\mathcal{V}_2=\{U^2_i=U_i\bigcap X_2 : X_2\setminus
U_i\neq\emptyset, i\in \mathbb{N} \}$ as families of subsets of
the space $X$. Define $\mathcal{V}:=\{ U^1_i\bigcap U^2_i :
U^1_i\in \mathcal{V}_1$ and $U^2_i\in \mathcal{V}_2 \}$. Note that
$\mathcal{V}$ is a countable $k$-cover  of $X$ by cozero sets. By
Theorem 18 in \cite{cmkm}, there is $\{U^1_{i_n}\bigcap U^2_{i_n}
:n\in \mathbb{N}\}\subset \mathcal{V}$ such that
$\{U^1_{i_n}\bigcap U^2_{i_n} :n\in \mathbb{N}\}$ is a
$\gamma_k$-cover of $X$. It follows that $\{U_{i_n}:n\in
\mathbb{N} \}$ is a $\gamma_k$-cover of $X\bigsqcup X$.

\end{proof}

\begin{theorem} For Tychonoff space $X$ the following statements are
equivalent:
\begin{enumerate}

\item  $C_k(X)$ is strongly sequentially separable;

\item  $(C_k(X))^n$ is strongly sequentially separable for each
$n\in \mathbb{N}$.

\end{enumerate}
\end{theorem}

\begin{proof} By Theorem \ref{th2}, Proposition \ref{pr1} and the argument
that $C_k(X\bigsqcup X)=C_k(X)\times C_k(X)$.

\end{proof}

\medskip

A.J. March  considered the problem (Problem 116 in \cite{ma}): Is
it possible to find spaces $X$, $Y$ such that $C_k(X)$ and
$C_k(Y)$ are strongly sequentially separable but $C_k(X)\times
C_k(Y)$ is not strongly sequentially separable ?

\medskip

 A.Miller constructed the following example \cite{mill}.

\begin{example} There exist disjoint subsets of the plane $X$ and $Y$ such that
both $X$ and $Y$ are $\gamma_k$-sets but $X\cup Y$ is not. Let $X$
be the open disk of radius one, i.e., $X=\{(x, y) : x^2 + y^2 < 1
\}$, and $Y$ be any singleton on the boundary of $X$, e.g., $Y =
\{(1, 0)\}$.
\end{example}

Thus, we have the example of the subsets of the plane $X$ and $Y$
such that $C_k(X)$ and $C_k(Y)$ are strongly sequentially
separable, but $C_k(X\cup Y)$ is not.

\medskip

Note that (in contrast to the $C_p$-theory)  $C_k(X\cup Y)\neq
C_k(X)\times C_k(Y)$.

\medskip

In \cite{cmkm}, the authors considered the next problem  (Problem
21 in \cite{cmkm}) :  Is the class of $\gamma_k$-sets closed for
finite unions ?

\medskip

A particular answer to this problem and March's problem is the
following

\begin{theorem} Suppose that $X$ and $Y$ are $\gamma_k$-sets, $iw(X)=iw(Y)=\aleph_0$ and $Y$ is first-countable.
Then $X\bigsqcup Y$ is a $\gamma_k$-set.
\end{theorem}

\begin{proof} By Theorem \ref{th2}, $C_k(X)$ and $C_k(Y)$ are strongly sequentially
separable. Notice that each hemicompact space belong to the class
$S_1(\mathcal{K},\Gamma_k)$, and the converse holds for first
countable spaces \cite{mc}. It follows that $C_k(Y)$ is separable
metrizable (first countable) space. By Theorem 9 in \cite{glm},
$C_k(X)\times C_k(Y)$ is a strongly sequentially separable. Since
$C_k(X)\times C_k(Y)=C_k(X\bigsqcup Y)$ and, by Theorem \ref{th2},
we have that $X\bigsqcup Y$ is a $\gamma_k$-set.

\end{proof}

\begin{corollary} The product $C_k(X)\times C_k(Y)$ of strongly sequentially
separable space $C_k(X)$ and strongly sequentially separable
first-countable space $C_k(Y)$ belongs to the class of strongly
sequentially separable spaces.

\end{corollary}

\section{$S_{1}(\mathcal{D},\mathcal{D})$}

In \cite{koc2} it was shown that a Tychonoff space $X$ belongs to
the class $S_1(\mathcal{K},\mathcal{K})$ if and only if $C_k(X)$
has countable strong fan tightness (i.e. for each $f\in C_k(X)$,
$S_1(\Omega_f, \Omega_f)$ holds \cite{sak}).
\medskip

Lj.D.R. Ko$\check{c}$inac proved the next

\begin{theorem}(Theorem 6 in \cite{cmkm})\label{th25} For a first countable Tychonoff space $X$  the following statements are
equivalent:

\begin{enumerate}

\item  $C_k(X)$ is first countable;

\item  $C_k(X)$ has countable strong fan tightness;

\item  $C_k(X)$ has countable fan tightness;

\item  $X$ is locally compact Lindel$\ddot{o}$f space;

\item $X$ satisfies $S_{1}(\mathcal{K},\mathcal{K})$;

\item $X$ satisfies $S_{fin}(\mathcal{K},\mathcal{K})$;

\end{enumerate}

\end{theorem}

We consider the generalizations (Theorem \ref{th22} and Theorem
\ref{th27}) of the Theorem \ref{th25} to the class of Tychonoff
spaces with $iw(X)=\aleph_0$.

\begin{theorem}\label{th22} For a Tychonoff space $X$ with $iw(X)=\aleph_0$ the following statements are
equivalent:

\begin{enumerate}

\item  $C_k(X)$ satisfies $S_{1}(\mathcal{D},\mathcal{D})$;

\item  $X$ satisfies $S_{1}(\mathcal{K},\mathcal{K})$;

\item Each finite power of $X$ satisfies
$S_{1}(\mathcal{K},\mathcal{K})$;

\item  $C_k(X)$ satisfies $S_{1}(\Omega_{\bf 0}, \Omega_{\bf 0})$
[countable strong fan tightness];

\item  $C_k(X)$ satisfies $S_{1}(\mathcal{D}, \Omega_{\bf 0})$.

\end{enumerate}

\end{theorem}

\begin{proof} $(2)\Leftrightarrow(3)$. By Theorem 5 in \cite{mkm}.

$(2)\Leftrightarrow(4)$. By Theorem 2.2 in \cite{koc2}.

$(1)\Rightarrow(2)$. Let $\mathcal{K}_i\in \mathcal{K}$ for every
$i\in \mathbb{N}$ and $D$ be a countable dense subset of $C_k(X)$.
Consider $D_i=\{f_{K,U,d}\in C(X) : f\vert (X\setminus U)\equiv 1$
and $f\vert K=d$ where $K$ is a compact subset of $X$, $U\in
\mathcal{K}_i$ such that $K\subset U$ and $d\in D \}$. Since $D$
is dense subset of $C_k(X)$, we have that $D_i$ is dense subset of
$C_k(X)$ for every $i\in \mathbb{N}$. By $(1)$, there is a
sequence $\{f_{K_i,U_i,d_i}\}_{i\in\mathbb{N}}$ such that for each
$i$, $f_{K_i,U_i,d_i}\in D_{i}$, and $\{f_{K_i,U_i,d_i} :
i\in\mathbb{N}\}$ is a dense subset of $C_k(X)$. Note that $U_i\in
\mathcal{K}_i$ for each $i\in \mathbb{N}$ and $\{U_i : i\in
\mathbb{N}\}\in \mathcal{K}$.

$(2)\Rightarrow(1)$. Let $(D_{i,j} : i,j\in \mathbb{N})$ be a
sequence of dense subsets of $C_k(X)$ and let $D=\{d_n : n\in
\mathbb{N}\}$ be a countable dense subset of $C_k(X)$. For every
couple $(i,j)$, $i,j\in \mathbb{N}$ and $f\in D_{i,j}$ consider
$K_{i,j,f}=\{x\in X : |f(x)-d_j(x)|<\frac{1}{i}\}$ and
$\mathcal{K}_{i,j}=\{K_{i,j,f} :  f\in D_{i,j}\}$. Claim that
$\mathcal{K}_{i,j}\in \mathcal{K}$ for every couple $(i,j)$,
$i,j\in \mathbb{N}$. Let $K\in \mathbb{K}(X)$ and $<d_j,K,
\frac{1}{i}>$ a base neighborhood of $d_j$. Since $D_{i,j}$ is a
dense subset of $C_k(X)$, there is $f\in D_{i,j}$ such that $f\in
<d_j,K, \frac{1}{i}>$, hence, $K\subset K_{i,j,f}$. Fix $j\in
\mathbb{N}$, by (2), there is a family $\{K_{i,j,f(i,j)}: i \in
\mathbb{N}\}$ such that $K_{i,j,f(i,j)}\in \mathcal{K}_{i,j}$ and
$\{K_{i,j,f(i,j)}: i\in\mathbb{N}\}\in \mathcal{K}$. So $f(i,j)\in
D_{i,j}$ for $i,j\in \mathbb{N}$. Claim that $\{f(i,j): i,j\in
\mathbb{N}\}$ is dense in $C_k(X)$. Let $p\in C(X)$, $K\in
\mathbb{K}(X)$, $\epsilon>0$ and let $<p,K, \epsilon>$ be a base
neighborhood of $p$. There is $j'\in \mathbb{N}$ such that
$d_{j'}\in <p,K, \frac{\epsilon}{2}>$. Since $\{K_{i,j',f(i,j')}:
i\in\mathbb{N}\}\in \mathcal{K}$, there is $i'\in \mathbb{N}$ such
that $K\subset K_{i',j',f(i',j')}$ and
$\frac{1}{i'}<\frac{\epsilon}{2}$. It follows that
$|f(i',j')(x)-p(x)|<|f(i',j')(x)-d_{j'}(x)|+|d_{j'}(x)-p(x)|<\frac{\epsilon}{2}+\frac{\epsilon}{2}=\epsilon$
for every $x\in K$. Hence, $f(i',j')\in <p,K, \epsilon>$ and
$\{f(i,j): i,j\in \mathbb{N}\}$ is dense in $C_k(X)$.

$(4)\Rightarrow(5)$ is immediate.

$(5)\Rightarrow(1)$. Let $(D_{i,j} : i\in \mathbb{N})$ be a
sequence of dense subsets of $C_k(X)$ for each $j\in \mathbb{N}$
and let $D=\{d_j: j\in \mathbb{N}\}$ be a countable dense subset
of $C_k(X)$. By (5), for every $j\in \mathbb{N}$ there is a family
$\{d^i_j: i\in\mathbb{N}\}$ such that $d^i_j\in D_{i,j}$ and
$\{d^i_j: i\in\mathbb{N}\}\in \Omega_{d_j}$. Note that $\{d^i_j:
i,j\in\mathbb{N}\}\in \mathcal{D}$.

\end{proof}

\section{$S_{fin}(\mathcal{D},\mathcal{D})$}

According to \cite{llt} $X$ belongs to
$S_{fin}(\mathcal{K},\mathcal{K})$ if and only if $C_k(X)$ has
countable fan tightness (i.e., for each $f\in C_k(X)$,
$S_{fin}(\Omega_{f}, \Omega_{f})$ holds \cite{arh1}).

\medskip
\begin{theorem}\label{th27} For a Tychonoff space $X$ with $iw(X)=\aleph_0$ the following statements are
equivalent:

\begin{enumerate}

\item  $C_k(X)$ satisfies $S_{fin}(\mathcal{D},\mathcal{D})$;

\item  $X$ satisfies $S_{fin}(\mathcal{K},\mathcal{K})$;

\item Each finite power of $X$ satisfies
$S_{fin}(\mathcal{K},\mathcal{K})$.

\item  $C_k(X)$ satisfies $S_{fin}(\Omega_{\bf 0}, \Omega_{\bf
0})$ [countable fan tightness];

\item  $C_k(X)$ satisfies $S_{fin}(\mathcal{D}, \Omega_{\bf 0})$.

\end{enumerate}

\end{theorem}

\begin{proof}

$(2)\Leftrightarrow(3)$. By Theorem 6 in \cite{mkm}.

$(2)\Leftrightarrow(4)$ see in \cite{llt}.

The remaining implications are proved similarly to the proof of
Theorem \ref{th22}.

\end{proof}

\begin{remark} It is easy to see that every hemicompact space is
in the class $S_1(\mathcal{K}, \mathcal{K})$ and, thus, in
$S_{fin}(\mathcal{K}, \mathcal{K})$. By Proposition 5 in
\cite{cmkm}, the converse is also true in the class of first
countable spaces.
\end{remark}

\begin{corollary}\label{th27} For a first countable Tychonoff space $X$  the following statements are
equivalent:

\begin{enumerate}

\item $C_k(X)$ satisfies $S_{1}(\mathcal{D},\mathcal{D})$;

\item $C_k(X)$ satisfies $S_{fin}(\mathcal{D},\mathcal{D})$;

\item $X$ satisfies $S_{1}(\mathcal{K},\mathcal{K})$.

\end{enumerate}

\end{corollary}

\section{$S_{1}(\mathcal{S},\mathcal{D})$}

\begin{definition}
A $\gamma_k$-cover $\mathcal{U}$ of co-zero sets of $X$  is {\bf
$\gamma_k$-shrinkable} if there exists a $\gamma_k$-cover $\{F(U)
: U\in \mathcal{U}\}$ of zero-sets of $X$ with $F(U)\subset U$ for
every $U\in \mathcal{U}$.
\end{definition}

For a topological space $X$ we denote:

$\bullet$ $\Gamma^{sh}_k$ --- the family of $\gamma_k$-shrinkable
$\gamma_k$-covers of $X$.

\begin{theorem}\label{th05} For a Tychonoff space $X$  the following statements are
equivalent:

\begin{enumerate}

\item  $C_k(X)$ satisfies $S_{1}(\Gamma_{\bf 0},\Omega_{\bf0})$;

\item  $X$ satisfies $S_{1}(\Gamma^{sh}_k,\mathcal{K})$.

\end{enumerate}

\end{theorem}

\begin{proof} $(1)\Rightarrow(2)$. Let $C_k(X)$ satisfies $S_{1}(\Gamma_{\bf 0},\Omega_{\bf0})$ and
$\{\mathcal{F}_i : i\in \mathbb{N}\}\subset \Gamma^{sh}_k$.

 For each $i\in \mathbb{N}$ we consider a
set $D_i=\{f_{F(U),U,i}\in C(X) : f_{F(U),U,i}\upharpoonright
F(U)=0$ and $f_{F(U),U,i}\upharpoonright (X\setminus U)=1$ for
$U\in \mathcal{F}_i \}$.

Since $\{F(U) : U\in \mathcal{F}_i\}$ is a $\gamma_k$-cover of
$X$, we have that $D_i$ converge to $f\equiv{\bf 0}$ for each
$i\in \mathbb{N}$.

Since $C_k(X)$ satisfies $S_{1}(\Gamma_{\bf 0},\Omega_{\bf 0})$,
there is a sequence $\{f_{F(U_i),U_i,i} : i\in\mathbb{N}\}$ such
that for each $i$, $f_{F(U_i),U_i,i}\in D_i$, and
$\{f_{F(U_i),U_i,i}: i\in\mathbb{N} \}$ is an element of
$\Omega_{\bf 0}$.

Consider $\{U_i: i\in \mathbb{N}\}$.

(a). $U_{i}\in \mathcal{F}_{i}$.

(b). $\{U_{i}: i\in \mathbb{N}\}$ is a $k$-cover of $X$.

Let $K$ be a non-empty compact subset of $X$ and $U=<f, K,
\frac{1}{2}>$ be a base neighborhood of $f$, then there is
$f_{F(U_{i'}),U_{i'},i'}\in U$. It follows that $K\subset U_{i'}$.
We thus get $X$ satisfies $S_{1}(\Gamma^{sh}_k,\mathcal{K})$.

$(2)\Rightarrow(1)$. Let $\{f_{k,i}\}_{k\in \mathbb{N}}$ be a
sequence converge to $f$ for each $i\in \mathbb{N}$. Without loss
of generality we can assume that $f=\bf{0}$, a set
$W^{i}_{k}=\{x\in X: -\frac{1}{i}< f_{k,i}(x)< \frac{1}{i}\}\neq
X$ for any $i\in \mathbb{N}$ and $S^{i}_{k}=\{x\in X:
-\frac{1}{i}\leq f_{k,i}(x)\leq \frac{1}{i}\}\neq X$ for any $i\in
\mathbb{N}$.

Consider $\mathcal{V}_i=\{W^{i}_{k} : k\in \mathbb{N}\}$ and
$\mathcal{S}_i=\{S^{i}_{k} : k\in \mathbb{N}\}$ for each $i\in
\mathbb{N}$. Claim that $\mathcal{V}_i$ is a $\gamma_k$-cover of
$X$. Since $\{f_{k,i}\}_{k\in \mathbb{N}}$ converge to $f$, for
each compact subset $K\subset X$ there is $k_0\in \mathbb{N}$ such
that $f_{k,i}\in <f, K, \frac{1}{i}
>$ for $k>k_0$. It follows that $K\subset W^{i}_{k}$ for any
$k>k_0$. Since $\mathcal{V}_{i+1}$ is a $\gamma_k$-cover,
$\mathcal{S}_{i+1}$ is a $\gamma_k$-cover, too.
$\mathcal{S}_{i+1}$ is a refinement of the family
$\mathcal{V}_{i}$, hence, $\mathcal{V}_{i}\in \Gamma^{sh}_k$.

By $X$ satisfies $S_{1}(\Gamma^{sh}_k,\mathcal{K})$, there is a
sequence $\{W^{i}_{k(i)}\}_{ i\in\mathbb{N}}$ such that for each
$i$, $W^{i}_{k(i)}\in \mathcal{V}_i$, and $\{W^{i}_{k(i)}:
i\in\mathbb{N} \}$ is an element of $\mathcal{K}$.

We claim that $f\in \overline{\{f_{k(i),i} : i\in \mathbb{N} \}}$.
Let $U=<f, K, \epsilon>$ be a base neighborhood of $f$ where
$\epsilon>0$ and $K\in \mathbb{K}(X)$, then there are $i_0, i_1\in
\mathbb{N}$ such that $\frac{1}{i_0}<\epsilon$, $i_1>i_0$ and
$W^{i_1}_{k(i_1)}\supset K$. It follows that $f_{k(i_1),i_1}\in
<f, K, \epsilon>$ and, hence, $f\in \overline{\{f_{k(i),i} : i\in
\mathbb{N} \}}$.

\end{proof}


\begin{lemma}\label{lem} Let $\mathcal{U}=\{U_n:n\in \mathbb{N}\}$ be a
$\gamma_k$-shrinkable co-zero cover of a space $X$. Then the set
$S=\{f\in C(X): f\upharpoonright (X\setminus U_n)\equiv 1$ for
some $n\in \mathbb{N}\}$ is sequentially dense in $C_k(X)$.
\end{lemma}

\begin{proof} Let $h\in C(X)$. For each $n\in \mathbb{N}$, take
$f_n\in C(X)$ such that $f_n\upharpoonright
F(U_n)=h\upharpoonright F(U_n)$ and $f_n\upharpoonright
(X\setminus U_n)\equiv 1$. Then obviously $f_n\in S$, and
$f_n\mapsto h$, because $\{F(U_n): n\in \mathbb{N}\}$ is a
$\gamma_k$-cover.
\end{proof}

\begin{theorem}\label{th5} For a Tychonoff space $X$ with $iw(X)=\aleph_0$ the following statements are
equivalent:

\begin{enumerate}

\item  $C_k(X)$ satisfies $S_{1}(\mathcal{S},\mathcal{D})$;

\item  $C_k(X)$ satisfies $S_{1}(\mathcal{S},\Omega_{\bf0})$;

\item  $C_k(X)$ satisfies $S_{1}(\Gamma_{\bf 0},\Omega_{\bf0})$;

\item  $X$ satisfies $S_{1}(\Gamma^{sh}_k,\mathcal{K})$.

\end{enumerate}

\end{theorem}

\begin{proof}

$(1)\Rightarrow(4)$. Let $\{\mathcal{F}_i : i\in
\mathbb{N}\}\subset \Gamma^{sh}_k$. By Lemma \ref{lem},
$S_i=\{f\in C(X): f\upharpoonright (X\setminus F^{i}_n)\equiv 1$
for some $F^{i}_n\in \mathcal{F}_i\}$ is a sequentially dense
subset of $C_k(X)$ for each $i\in \mathbb{N}$.

By (1), there is $\{f_i : i\in \mathbb{N}\}$ such that $f_i\in
S_i$ and $\{f_i : i\in \mathbb{N}\}\in \mathcal{D}$.

Consider a sequence $\{F^{i}_{n(i)} : i\in \mathbb{N}\}$.

(a). $F^{i}_{n(i)}\in \mathcal{F}_{i}$ for $i\in \mathbb{N}$.

(b). $\{F^{i}_{n(i)} : i\in \mathbb{N}\}$ is a $k$-cover of $X$.

Let $K\in \mathbb{K}(X)$ and let $U=<$ $\bf{0}$ $, K,
\frac{1}{2}>$ be a base neighborhood of $\bf{0}$, then there is
$f_{i'}\in \{f_i : i\in \mathbb{N}\}$ such that $f_{i'}\in U$. It
follows that $K\subset F^{i'}_{n(i')}$.

$(4)\Rightarrow(3)$. Let $X$ satisfies
$S_{1}(\Gamma^{sh}_k,\mathcal{K})$ and let $\{f_{i,m}\}_{m\in
\mathbb{N}}$ converge to $\bf{0}$ for each $i\in \mathbb{N}$.

Consider $\mathcal{F}_i=\{F_{i,m} : m\in
\mathbb{N}\}=\{f^{-1}_{i,m}(-\frac{1}{i}, \frac{1}{i}): m\in
\mathbb{N} \}$ for each $i\in \mathbb{N}$. Without loss of
generality we can assume that a set $F_{i,m}\neq X$ for any
$i,m\in \mathbb{N}$. Otherwise there is sequence
$\{f_{i_k,m_k}\}_{k\in \mathbb{N}}$ such that
$\{f_{i_k,m_k}\}_{k\in \mathbb{N}}$ uniform converge to $\bf{0}$
and $\{f_{i_k,m_k}: k\in \mathbb{N}\}\in \Omega_{\bf 0}$.

Note that $\mathcal{F}_i$ is a $\gamma_k$-shrinkable co-zero cover
of $X$ for each $i\in \mathbb{N}$.

By (4), there is a sequence $(F_{i,m(i)}: i\in\mathbb{N})$ such
that for each $i$, $F_{i,m(i)}\in \mathcal{F}_i$, and
$\{F_{i,m(i)}: i\in\mathbb{N} \}$ is an element of $\mathcal{K}$.

We claim that $\bf 0$ $\in \overline{\{f_{i,m(i)} : i\in
\mathbb{N} \}}$. Let $W=<$ $\bf 0$ $, K, \epsilon>$ be a base
neighborhood of $\bf 0$ where $\epsilon>0$ and $K\in
\mathbb{K}(X)$, then there are $i_0, i_1\in \mathbb{N}$ such that
$\frac{1}{i_0}<\epsilon$, $i_1>i_0$ and $F_{i_1,m(i_1)}\supset K$.
It follows that $f_{i_1, m(i_1)}\in <$ $\bf 0$ $, K, \epsilon>$
and, hence, $\bf 0$ $\in \overline{\{f_{i,m(i)} : i\in \mathbb{N}
\}}$ and $C_k(X)$ satisfies $S_{1}(\Gamma_{\bf 0},\Omega_{\bf
0})$.

$(3)\Rightarrow(2)$ is immediate.

$(2)\Rightarrow(1)$. Suppose that $C_k(X)$ satisfies
$S_{1}(\mathcal{S},\Omega_{\bf0})$. Let $D=\{d_n: n\in \mathbb{N}
\}$ be a dense subspace of $C_k(X)$. Given a sequence of
sequentially dense subspaces of $C_k(X)$, enumerate it as
$\{S_{n,m}: n,m \in \mathbb{N} \}$. For each $n\in \mathbb{N}$,
pick $d_{n,m}\in S_{n,m}$ so that $d_n\in \overline{\{d_{n,m}:
m\in \mathbb{N}\}}$. Then $\{d_{n,m}: m,n\in \mathbb{N}\}$ is
dense in $C_k(X)$.
\end{proof}

\section{$S_{fin}(\mathcal{S},\mathcal{D})$}

The following Theorems are proved similarly to Theorems \ref{th05}
and \ref{th5}.

\begin{theorem}\label{th06} For a Tychonoff space $X$  the following statements are
equivalent:

\begin{enumerate}

\item  $C_k(X)$ satisfies $S_{fin}(\Gamma_{\bf 0},\Omega_{\bf0})$;

\item  $X$ satisfies $S_{fin}(\Gamma^{sh}_k,\mathcal{K})$.

\end{enumerate}

\end{theorem}

\begin{theorem}\label{th6} For a Tychonoff space $X$ with $iw(X)=\aleph_0$ the following statements are
equivalent:

\begin{enumerate}

\item  $C_k(X)$ satisfies $S_{fin}(\mathcal{S},\mathcal{D})$;

\item  $C_k(X)$ satisfies $S_{fin}(\mathcal{S},\Omega_{\bf0})$;

\item  $C_k(X)$ satisfies $S_{fin}(\Gamma_{\bf 0},\Omega_{\bf0})$;

\item  $X$ satisfies $S_{fin}(\Gamma^{sh}_k,\mathcal{K})$.

\end{enumerate}

\end{theorem}

\section{$S_{1}(\mathcal{S},\mathcal{S})$}

In \cite{os2}, we proved the following theorems.

\begin{theorem} (Theorem 3.3 in \cite{os2})\label{th07} For a Tychonoff space $X$  the following statements are
equivalent:

\begin{enumerate}

\item  $C_k(X)$ satisfies $S_{1}(\Gamma_{\bf 0},\Gamma_{\bf 0})$;

\item  $X$ satisfies $S_{1}(\Gamma^{sh}_k,\Gamma_k)$.

\end{enumerate}

\end{theorem}

\begin{theorem}(Theorem 3.5 in \cite{os2})\label{th7} For a Tychonoff space $X$ such that $C_k(X)$ is sequentially separable  the following statements are
equivalent:

\begin{enumerate}

\item  $C_k(X)$ satisfies $S_{1}(\mathcal{S},\mathcal{S})$;

\item  $C_k(X)$ satisfies $S_{1}(\mathcal{S},\Gamma_{\bf 0})$;

\item  $C_k(X)$ satisfies $S_{1}(\Gamma_{\bf 0},\Gamma_{\bf 0})$;

\item  $X$ satisfies $S_{1}(\Gamma^{sh}_k,\Gamma_k)$;

\item $C_k(X)$ satisfies $S_{fin}(\mathcal{S},\mathcal{S})$;

\item  $C_k(X)$ satisfies $S_{fin}(\mathcal{S},\Gamma_{\bf 0})$;

\item  $C_k(X)$ satisfies $S_{fin}(\Gamma_{\bf 0},\Gamma_{\bf
0})$;

\item  $X$ satisfies $S_{fin}(\Gamma^{sh}_{k},\Gamma_{k})$.

\end{enumerate}

\end{theorem}

\medskip

We can summarize the relationships between considered notions in
next diagrams.

\medskip
\begin{center}
$S_1(\mathcal{S},\mathcal{S}) \Leftrightarrow
S_{fin}(\mathcal{S},\mathcal{S}) \Rightarrow
S_1(\mathcal{S},\mathcal{D}) \Rightarrow
S_{fin}(\mathcal{S},\mathcal{D})$ \\  \, \, $\Uparrow$ \, \, \, \,
\,\, \,  \, \, \, $ \Uparrow $ \,\, \, \, \,\, \, \,
\, \, \, $\Uparrow $ \,\, \, \, \,\, \, \,\, \, \, $\Uparrow$ \\
$S_1(\mathcal{D},\mathcal{S}) \Leftrightarrow
S_{fin}(\mathcal{D},\mathcal{S}) \Rightarrow
S_1(\mathcal{D},\mathcal{D}) \Rightarrow
S_{fin}(\mathcal{D},\mathcal{D})$

\medskip

Diagram~1. The Diagram of selectors for sequences of dense sets of
$C_k(X)$.

\end{center}

\bigskip

\medskip
\begin{center}
$S_1(\Gamma^{sh}_k,\Gamma_k) \Leftrightarrow
S_{fin}(\Gamma^{sh}_k,\Gamma_k) \Rightarrow
S_1(\Gamma^{sh}_k,\mathcal{K}) \Rightarrow
S_{fin}(\Gamma^{sh}_k,\mathcal{K})$ \\  \, \, $\Uparrow$ \, \, \,
\, \,\, \,  \, \, \, $ \Uparrow $ \,\, \, \, \,\, \, \,
\, \, \, $\Uparrow $ \,\, \, \, \,\, \, \,\, \, \, $\Uparrow$ \\
$S_1(\mathcal{K},\Gamma_k) \Leftrightarrow
S_{fin}(\mathcal{K},\Gamma_k) \Rightarrow
S_1(\mathcal{K},\mathcal{K}) \Rightarrow
S_{fin}(\mathcal{K},\mathcal{K})$

\medskip

Diagram~2. The Diagram of selection principles for a space $X$
corresponding to selectors for sequences of dense sets of
$C_k(X)$.

\end{center}

\bigskip

\section{On the particular solution to one problem}

Recall that Arens's space $S_2$ is the set $\{(0,0),
(\frac{1}{n},0), (\frac{1}{n}, \frac{1}{nm}) : n,m \in
\mathbb{N}\setminus\{0\}\} \}\subset \mathbb{R}^2$ carrying the
strongest topology inducing the original planar topology on the
convergent sequences $C_0 = \{(0, 0), ( \frac{1}{n}, 0) : n
> 0\}$ and $C_n = \{( \frac{1}{n}, 0), (\frac{1}{n}, \frac{1}{nm}) : m > 0\}$, $n > 0$. The
sequential fan is the quotient space $S_{\omega}= S_2/C_0$
obtained from the Arens's space by identifying the points of the
sequence $C_0$ \cite{l}.

\begin{proposition}\label{pr1} If $C_k(X)$  satisfies $S_{fin}(\Gamma_{\bf
0},\Omega_{\bf0})$, then $S_{\omega}$ cannot be embedded into
$C_k(X)$.
\end{proposition}

The following problem was posed in the paper \cite{cmkm}.

\medskip

{\bf Problem 30}. Does a first countable (separable metrizable)
space belong to the class $S_1(\Gamma_k,\mathcal{K})$ if and only
if it is hemicompact?

\medskip

A particular answer to this problem  is the following

\begin{theorem} Suppose that $X$ is first countable stratifiable
space and $iw(X)=\aleph_0$. Then following the statements are
equivalent:

\begin{enumerate}

\item  $X$ satisfies $S_{fin}(\Gamma^{sh}_k,\mathcal{K})$;

\item  $X$ satisfies $S_{fin}(\Gamma_k,\mathcal{K})$;

\item  $X$ satisfies $S_{1}(\mathcal{K},\Gamma_k)$;

\item  $X$ is hemicompact.

\end{enumerate}

\end{theorem}

\begin{proof} $(1)\Rightarrow(4)$. Since $X$ is first countable stratifiable
space and, by Proposition \ref{pr1},  $S_{\omega}$ cannot be
embedded into $C_k(X)$, then, by Theorem 2.2 (+ Remark) in
\cite{gtz}, $X$ is a locally compact. A locally compact
stratifiable space is metrizable \cite{ced}. A well-known that a
locally compact metrizable space can be represented as
$X=\bigsqcup\limits_{\alpha<\tau} X_{\alpha}$ where $X_{\alpha}$
is a $\sigma$-compact for each $\alpha<\tau$. Since
$iw(X)=\aleph_0$, then $\tau\leq\mathfrak{c}$. Claim that
$|\tau|<\omega_1$.

Assume that $|\tau|\geq \omega_1$. Then there is a continuous
mapping $f:X \mapsto D$ from $X$ onto a discrete space $D$ where
$|D|\geq \omega_1$. It follows that $D$ satisfies
$S_{fin}(\Gamma^{sh}_k,\mathcal{K})$ ($S_{fin}(\Gamma,\Omega)$)
and, hence, $D$ is Lindel$\ddot{o}$f, a contradiction.

It follows that $X$ is a locally compact and Lindel$\ddot{o}$f,
and, hence, $X$ is a hemicompact.

$(4)\Rightarrow(3)$. Since $X$ is hemicompact and
$iw(X)=\aleph_0$, then $C_k(X)$ is a separable metrizable space
\cite{mcnt}. Hence, $C_k(X)$ satisfies
$S_{1}(\mathcal{D},\mathcal{S})$, and, by Theorem \ref{th1}, $X$
satisfies $S_{1}(\mathcal{K},\Gamma_k)$.
\end{proof}

\begin{corollary} Suppose that $X$ is separable metrizable space. Then $X$ satisfies $S_{fin}(\Gamma_k,\mathcal{K})$ if and only if
$X$ is hemicompact.

\end{corollary}

\begin{remark} In class of first countable stratifiable
spaces with $iw(X)=\aleph_0$ (in particular, in class of separable
metrizable spaces) all properties in Diagram~1 (and, hence,
Diagram~2) coincide.

\end{remark}


\bibliographystyle{model1a-num-names}
\bibliography{<your-bib-database>}







\end{document}